\documentclass[runningheads,envcountsame,a4paper]{llncs} 

\usepackage{amssymb,amsmath}
\usepackage{wasysym}
\usepackage{graphicx}
\usepackage{epsfig}
\usepackage{latexsym}
\usepackage[english]{babel}
\usepackage[utf8]{inputenc}
\usepackage{csquotes}
\usepackage{algorithm} 
\usepackage{algpseudocode}
\usepackage{tikz}
\usetikzlibrary{arrows,shapes.geometric,positioning,calc}
\usepackage[colorinlistoftodos]{todonotes}
\usepackage{changepage}
\usepackage{subfig}
\usepackage{soul}
\usepackage{hyperref}
\usepackage{todonotes}
\usepackage{enumerate}
\usepackage{authblk}

\newtheorem{thm}{Theorem}
\newtheorem{prop}[thm]{Proposition}
\newtheorem{cor}[thm]{Corollary}
\newtheorem{lem}[thm]{Lemma}

\newtheorem{exa}[thm]{Example}

\newsavebox{\qedB}

\sbox{\qedB}{\setlength{\unitlength}{1mm}
 \begin{picture}(4,4)(0,0)
  \thinlines
  {\put(0,0){\framebox(2.83,2.83){}}}%
  {\put(1.17,1.17){\framebox(2.83,2.83){}}}%
  {\put(0,0){\framebox(4,4){}}}%
  {\put(1.17,1.17){{\rule{1ex}{1ex} }}}%
 \end{picture}}
 
\newcommand{\QEDB}{\ifmmode\def\next{\tag"\usebox{\qedB}"}%
 \else\let\next=\relax
 {\unskip\nobreak\hfil\penalty50
 \hskip2em\hbox{}\nobreak\hfil\usebox{\qedB}
 \parfillskip=0pt \finalhyphendemerits=0\penalty-100\bigskip}\fi\next}

\newcommand{\N}{\mathbb N}

\newcommand{\bprop}{\begin{prop}}
\newcommand{\eprop}{\end{prop}}
\newcommand{\bcor}{\begin{cor}}
\newcommand{\ecor}{\end{cor}}
\newcommand{\blem}{\begin{lem}}
\newcommand{\elem}{\end{lem}}

\newcommand{\pset}{[\![p]\!]}
\newcommand{\mset}{[\![m]\!]}
\newcommand{\Mod}{\!\mod\!}

\title{Block-counting sequences are not purely morphic}
\toctitle{Block-counting}

\author{Antoine Abram\inst{1} \and Yining Hu\inst{2} \and Shuo Li\inst{1}}

\institute{Laboratoire de Combinatoire et d'Informatique Mathématique,\\
Université du  Québec \`a Montréal,
 Montréal (QC), Canada,\\
\email{abram.antoine@lacim.ca, li.shuo@lacim.ca} 
\and
Institute for Advanced Study in Mathematics\\
Harbin Institute of Technology, Harbin, PR China\\
\email{huyining@protonmail.com}
}

\begin{document}

\maketitle


\begin{abstract}

Let $m$ be a positive integer larger than $1$, let $w$ be a finite word over $\left\{0,1,...,m-1\right\}$ and let $a_{m;w}(n)$ be the number of occurrences of the word $w$ in the $m$-expansion of $n$ mod $p$ for any non-negative integer $n$. In this article, we first give a fast algorithm to generate all sequences of the form $(a_{m;w}(n))_{n \in \mathbf{N}}$; then, under the hypothesis that $m$ is a prime, we prove that all these sequences are $m$-uniformly but not purely morphic, except for $w=1,2,...,m-1$; finally, under the same assumption of $m$ as before, we prove that the power series $\sum_{i=0}^{\infty} a_{m;w}(n)t^n$ is algebraic of degree $m$ over $\mathbb{F}_m(t)$. 
\end{abstract}

\section{Introduction, definitions and notation}

Given a positive integer $m$ larger than $1$ and a finite word $w$ over $\left\{0,1,2,...,m-1\right\}$, the block-counting sequence $(e_{m;w}(n))_{n \in \mathbf{N}}$ counts the number of occurrences of the word $w$ in the $m$-expansion of $n$ for each non-negative integer $n$. Let us define $(a_{m;w}(n))_{n \in \mathbf{N}}$ to be a sequence over $\left\{0,1,2,...,m-1\right\}$ such that $a_{m;w}(n) \equiv e_{m;w}(n) \mod (m)$ for all non-negative integer $n$. The analytical as well as the combinatorial properties of these sequences have been studied since 1900's after Thue and some well-known sequences are strongly related to this notion.  Recall that the $0,1$-Thue-Morse sequence can be defined as $(a_{2;1}(n))_{n \in \mathbf{N}}$ (see, for example, Page 15 in \cite{allouche_shallit_2003}) and the $0,1$-Rudin-Shapiro sequence can also be defined as $(a_{2;11}(n))_{n \in \mathbf{N}}$ (see, for example, Example 3.3.1 in \cite{allouche_shallit_2003}). In this article, we review some common properties of usual block-counting sequences and generalize them to all block-counting sequences.

To be able to announce our results, here we recall some definitions and notation. Let $A$ be a finite set. It will be called an {\em alphabet} and its elements will be called {\em letters}. Let $A^*$ denote the free monoid generated by $A$ under concatenations and let $A^{\mathbf{N}}$ denote the set of infinite concatenations of elements in $A$. Let $A^{\infty}=A^* \cup A^{\mathbf{N}}$. A {\em finite word} over the alphabet $A$ is an element in $A^*$ and an {\em infinite word} over $A$ is an element in $A^{\mathbf{N}}$. Particularly, the empty word is an element in $A^*$ and it is denoted by $\epsilon$. The length of a word $w$, denoted by $|w|$, is the number of letters that it contains. The length of the empty word is $0$ and the length of any infinite word is infinite. For any non-empty word $w \in A^{\infty}$, it can be denoted by $w[0]w[1]w[2]...$, where $w[i]$ are elements in $A$. A word $x$ is called a {\em factor} of $w$ if there exist two integers $0 \leq i \leq j \leq|w|-1$ such that $x=w[i]w[i+1]...w[j]$, this factor can also be denoted by $w[i..j]$. A factor $x$ is called a {\em prefix} (resp. a {\em suffix} ) of the word $w$ if there exists a positive integer $i$ such that $0 \leq i \leq |w|$ and $x=w[0..i]$ (resp. $x=w[i..|w|-1]$). For any finite word $w$ and any positive integer $n$, let $w^n$ denote the concatenation of $n$ copies of $w$, i.e. $w^n=ww...w$ $n$ times. Particularly, $w^0=\epsilon$. For any pair of words $w,v$ such that $v$ is a factor of $w$, let $|w|_v$ denote the number of occurrences of $v$ in $w$. 

Let $A$ and $B$ be two alphabets, a {\em morphism} $\phi$ from $A$ to $B$ is a map from $A^{\infty}$ to $B^{\infty}$ satisfying $\phi(xy)=\phi(x)\phi(y)$ for any pair of elements $x, y$ in $A^{\infty}$. The morphism $\phi$ is called {\em $k$-uniform} if for all elements $a \in A$, $|\phi(a)|=k$ and it is called {\em non-uniform} otherwise. A morphism $\phi$ is called a {\em coding} function if it is $1$-uniform and it is called {\em non-erasing} if $\phi(a) \neq \epsilon$ for all $a \in A$.

Let $A$ be a finite alphabet and let $(a_n)_{n \in \mathbf{N}}$ be an infinite sequence over $A$, it is called {\em morphic} if there exists an alphabet $B$, an infinite sequence $(b_n)_{n \in \mathbf{N}}$ over $B$, a non-erasing morphism $\phi$ from $B^{\infty}$ to $B^{\infty}$ and a coding function $\psi$ from $B^{\infty}$ to $A^{\infty}$, such that $(b_n)_{n \in \mathbf{N}}$ is a fixed point of $\phi$ and $(a_n)_{n \in \mathbf{N}}=\psi((b_n)_{n \in \mathbf{N}})$. Moreover, the sequence $(a_n)_{n \in \mathbf{N}}$ is called {\em uniformly morphic} if $\phi$ is $k$-uniform for some integer $k$, and it is called {\em non-uniformly morphic} otherwise. The sequence $(a_n)_{n \in \mathbf{N}}$ is called {\em purely morphic} if $A=B$ and $\psi=Id$.

For any positive integer $m$, let $[\![m]\!]=\left\{0,1,2,...,m-1\right\}$. For any $t \in [\![m]\!]$ let $t^+ \equiv t+1 \mod m$; for any $w \in [\![m]\!]^*$, let $w^+=w[0]^+w[1]^+...w[|w|-1]^+$.

In Section\ref{sec:windows}, we give a fast algorithm to generate all block-counting sequences. It is well-known that the Thue-Morse sequence can be generated by the following algorithm (see, for example, \cite[A008277]{oeis}):

\begin{exa}
Let $(w_n)_{n \in \mathbf{N}}$ be a sequence of words over the $[\![2]\!]^*$ such that $w_0=0$ and that $w_{i+1}=w_iw_i^+$ for all $i$, then the Thue-Morse sequence $(a_{2;1}(n))_{n \in \mathbf{N}}$ satisfies $(a_{2;1}(n))_{n \in \mathbf{N}}= \lim_{i \to \infty}w_i$.
\end{exa}
In Section\ref{sec:windows}, we prove that the Rudin-Shapiro sequence can also be generalized by a similar algorithm, see \ref{exa_3}. More generally, we find fast algorithms to generate all block-counting sequences. These algorithms are given by \ref{prop:blockx} and \ref{prop:block0} in Section\ref{sec:windows}.

From the definitions recalled as above, any morphic word can be classified as either a uniformly morphic word or a non-uniformly morphic word. However, from a recent article \cite{Allouche2020}, Allouche and Shallit proved that all uniformly morphic sequences are also non-uniformly morphic. This result implies that all sequences in the family of morphic sequences are also in its subfamily of non-uniformly morphic sequences. Indeed, many works can be found in the literature in the direction of characterizing all those non-uniformly morphic sequences which are {\em not} uniformly morphic, for example, one can find \cite{ITA1989}\cite{shallit1996}\cite{BARTHOLDI2003}\cite{AIF2006}\cite{BARTHOLDI2010}\cite{Benli2012}\cite{ALLOUCHE202011}. However, in \cite{Allouche2020}, it is proved actually that all uniformly morphic sequences are also non-uniformly {\em non-purely} morphic. In other words, from the construction of the proof given in \cite{Allouche2020}, a nontrivial coding function is required. In Section \ref{sec:morphic}, we investigate all those uniformly morphic sequences which are not purely morphic. 
It is already known that the Rudin-Shapiro sequence is in this case (Example 26 in \cite{morphic}). In section \ref{sec:morphic}, we prove that all other sequences in the form of $(a_{m,w}(n))_{n \in \mathbf{N}}$ have the same property when $|w| \neq 1$ and $m$ is a prime.  The result is announced as follows:

\begin{thm}
\label{th1}
Let $p$ be a prime number and $w \in \pset^*$. The sequence $(a_{p;w}(n))_{n \in \mathbf{N}}$ is a $p$-uniformly morphic. Moreover, if $|w|=1$ and $w\neq 0$, this sequence is purely morphic and if not is it non-purely morphic.
\end{thm}

In Section \ref{sec:alg}, under the assumption that $p$ is a prime number, we prove that the formal power series $f_{p;w}=\sum_{i=0}^{\infty} a_{m;w}(n)t^n$ is algebraic and of degree $p$ over $\mathbb{F}_p(t)$. Indeed, from Christol's theorem \cite{Christol1980KMFR}, we know that the power series $f_{p;w}$ is algebraic over $\mathbb{F}_p(t)$. 
In Section \ref{sec:alg}, we prove that $f$ is algebraic of degree $p$.
\section{Windows functions and $(a_{p;w}(n))_{n \in \mathbf{N}}$}
\label{sec:windows}

For any positive integer $m$ and non-negative integer $n$, let $[n]_m$ denote the expansion of $n$ in the base $m$. For a given word $w \in \mset^*=\{0,1,\cdots, m-1\}^*$, $w=w[0]w[1]...w[|w|-1]$, let $(w)_m=\sum_{i=0}^{|w|-1}w[i]m^{|w|-1-i}$ and let $w'=w[1]w[2]\cdots w[|w|-1]$. A word $w$ is called a $x$-word if $w[0]=x$. For a given string $w$, let $\alpha_w=\frac{(w')_m}{m^{|w|-1}}$, $\beta_w=\frac{(w')_m+1}{m^{|w|-1}}$ and let $\phi_w: \mset^* \to \mset^*$ be a function such that for any $v \in  \mset^*$, $\phi_w(v)$ satisfies the following propriety:
$$
\phi_w(v)[i]=
\begin{cases}
v[i]+1 \Mod m \;\; \text{if }\alpha_w|v| \leq i <\beta_w|v|\\
v[i]\;\;\;\text{otherwise}.
\end{cases}
$$

\subsection{Block-counting sequences for non-$0$-words}

\begin{prop}
\label{prop:blockx}
Let $m$ a positive number, let $x\in\mset\backslash\{0\}$, let $w \in \mset^*$ be a $x$-word and let $t=(v)_m$. If we let $(u_i)_{i \in \N}$ be a sequence of words such that $|u_0|=m^{|w|}$, that
$$
u_0[i]=
\begin{cases}
1 \;\; \text{if i=t}\\
0\;\;\;\text{otherwise},
\end{cases}
$$
and that $u_{k+1}=u_k^x\phi_w(u_k)u_k^{m-x-1}$, then $\lim_{k \to \infty}u_k=(a_{m;w}(n))_{n \in \N}$.
\end{prop}

\begin{proof}
First, it is obvious that $u_0$ is a prefix of $(a_{m;w}(n))_{n \in \mathbf{N}}$.
Now let $y\in \mset\backslash\{0\}$.
For any integers $r$ and $m^k$ such that $0 \leq r < m^k$, $0 \leq e_{m;w}(r+ym^k)-e_{m;w}(r) \leq 1$.
Indeed, since $y\neq 0$, $[r+ym^k]_m=y0..0[r]_m$, thus, $[r+ym^k]_m$ has exactly one more $x$-factor of length $|v|$ than $[r]_p$ only if $y=x$, and this factor can be $w$ or not.
Moreover, $e_{m;w}(r+ym^k)-e_{m;w}(r)=1$ only if $w$ is a prefix of $[r+ym^k]_m$.
Consequently, $e_{m;w}(r+ym^k)=e_{m;w}(r)+1 $ only if $\alpha_wm^k \leq r <\beta_wm^k$ and $y=x$.
Hence, for any $t \in \mset\backslash\{x\}$,
\begin{align*}
(a_{m;w}(n))_{tm^k \leq n <(t+1) m^k}&=(a_{m;w}(n))_{0 \leq n <m^{k}}\\
(a_{m;w}(n))_{xm^k \leq n <(x+1) m^k}&=\phi_w((a_{m;w}(n))_{0 \leq n <m^{k}}).
\end{align*}

This implies that
 $$(a_{m;w}(n))_{0 \leq n <m^{k+1}}=u_k^{x}\phi_w(u_k)u_k^{m-x-1},$$ which concludes the proof.\qed
\end{proof}

\begin{exa}
\label{exa_3}
Let us compute the Rudin-Shapiro sequence using windows function. From Example 3.3.1 in \cite{allouche_shallit_2003}, the Rudin-Shapiro sequence can be defined as $(a_{2;11}(n))_{n \in \mathbf{N}}$. From Proposition \ref{prop:blockx}, set $\alpha_{11}=\frac{1}{2}$, $\beta_{11}=\frac{2}{2}$ and $s_0=0,0,0,1$. For any words $w \in \left\{0,1\right\}^*$ such that $w=w_1w_2$ with $|w_1|=|w_2|$, $\phi_s(w)=w_1(w_2^+)$. Thus, one can compute
$$s_1=0,0,0,1,0,0,1,0; \;\;\;\;\;\;\;\;\;\;\; s_2=0,0,0,1,0,0,1,0,0,0,0,1,1,1,0,1;$$
$$s_3=0,0,0,1,0,0,1,0,0,0,0,1,1,1,0,1,0,0,0,1,0,0,1,0,1,1,1,0,0,0,1,0;$$
$(e_s(n))_{n \in \mathbf{N}}$ is the limit of $s_n$ when $n$ tends to infinite.\qed
\end{exa}

\subsection{Block-counting sequences for $0$-words}

\begin{prop}
\label{prop:block0}
Let $m$ be a positive number, let $w\in\mset^*$ a $0$-word and let $t=(w)_m$. Let $u_0$ be such that $|u_0|=m^{|w|}$ and
$$
u_0[i]=
\begin{cases}
1 \;\; \text{if i=t}\\
0\;\;\;\text{otherwise},
\end{cases}
$$
and let $u_{k+1}=\phi_w(u_k)u_k^{m-1}$.

By letting $w_{-1}=u_0$ if $w=0^{|w|}$ and $w_{-1}=0^{m^{|w|}}$ if not, $w_k=u_k^{m-1}$ for $k\geq 0$, then $$(a_{m;w}(n))_{n \in \N}=w_{-1}w_0w_1w_2\cdots w_n\cdots.$$
\end{prop} 

\begin{lem}
\label{lem:0block}
Let $m$ be a positive number, $y\in\mset\backslash\{0\}$, $w\in\mset^*$ a $0$-word
and let $t=(w)_m$, then for any integer $r$ satisfying $t <m^k\leq r <m^{k+1}$:\\
1) $e_{m;w}(r+ym^{k+1})=e_{m;w}(r)$;\\
2) $0 \leq e_{m;w}(r+m^{k})-e_{m;w}(r) \leq 1$;\\
3) $e_{m;w}(r+m^{k})-e_{m;w}(r)=1$ only if $[r]_m$ is a $m-1$-word and $\alpha_wm^k \leq r <\beta_wm^k$.\\
\end{lem}

\begin{proof}
For any integer $r$ satisfying $t <m^k\leq r <m^{k+1}$, we first remark that $[r+ym^{k+1}]_m=y[r]_m$.
Since $[r+ym^{k+1}]_m$ and $y[r]_m$ have the same set of $0$-factors, $e_{m;w}(r+ym^{k+1})=e_{m;w}(r)$.
Second, if $[r]_m$ is not a $m-1$-word than $[r]_m$ and $[r+m^k]_m$ has the same set of $0$-factors. But if $[r]_m$ is a $m-1$-word, then $[r+m^k]_m=10[r]_m'$ and thus, can have at most one more $0$ factors of length $|w|$ than $[r]_m$.
Consequently, $0 \leq e_{m;w}(r+m^{k})-e_{m;w}(r) \leq 1$.
Moreover, in the latter case, $e_{m;w}(r+m^k)-e_{m;w}(r) =1$ only if $1w$ is a prefix of $[r+m^k]_m$.
So $e_{m;w}(r+m^k)=e_{m;w}(r)+1 $ only if $\alpha_wm^k <r \leq\beta_wm^k$.\qed
\end{proof}

\begin{proof}[of Proposition \ref{prop:block0}]
We first remark that $w_{-1}w_0$ is a prefix of $(a_{m;w}(n))_{n \in \mathbf{N}}$.

Further, for any integer $k$ satisfying $(w)_m <m^k$ and $x\in\mset\backslash\{0\}$, from Lemma \ref{lem:0block},
\begin{align*}
(a_{m;w}(n))_{xm^k \leq n <(x+1)m^k} &= (a_{m;w}(n))_{m^k \leq n <2m^k}\\
(a_{m;w}(n))_{m^{k+1} \leq n <m^{k+1}+m^k} &= (\phi(a_{m;w}(n))_{(p-1)m^k \leq n <m^{k+1}})).
\end{align*}
This implies that  
$$(a_{m;w}(n))_{m^k \leq n < m^{k+1}}=\left(\phi_w((a_{m;w}(n))_{m^{k-1} \leq n < m^{k}})(a_{m;w}(n))_{m^{k-1} \leq n < m^{k}}^{m-1}\right)^{m-1},$$
which concludes the proof.\qed
\end{proof}

\begin{exa}
Let us compute the sequence $(a_{2;01}(n))_{n \in \mathbf{N}}$ with. From the previous theorem, set $\alpha_{01}=\frac{1}{2}$, $\beta_{01}=\frac{2}{2}$, $s_{-1}=0,0,0,0$ and $s_0=0,1,0,0$. For any words $w \in \left\{0,1\right\}^*$ such that $w=w_1w_2$ with $|w_1|=|w_2|=k$ for some integer $k$, $\phi_s(w)=w_1w_2^+$. Thus, one can compute
$$s_1=0,1,1,1,0,1,0,0; \;\; s_2=0,1,1,1,1,0,1,1,0,1,1,1,0,1,0,0;$$
$$
s_3=0,1,1,1,1,0,1,1,1,0,0,0,1,0,1,1, 0,1,1,1,1,0,1,1,0,1,1,1,0,1,0,0;
$$
$(a_{2;01}(n))_{n \in \mathbf{N}}$ is the limit of $s_{-1}s_0s_1s_2s_3...s_n$ when $n$ tends to infinite.\qed
\end{exa}

\section{$(a_{p;w}(n))_{n \in \mathbf{N}}$ are not purely morphic}
\label{sec:morphic}

From now on, we work with $p$ a prime number.

We first prove that $(a_{p;w}(n))_{n \in \mathbf{N}}$ is not purely morphic when $|w|>1$.  We will need a simple notation, for $w=w[0]\cdots w[|w|-1]$, let $w^\diamond=w[0]\cdots w[|w|-2]$.

\begin{prop} 
\label{prop:block}
For any prime number $p$ and for any $w \in \pset^*$, the sub-sequences of the form $(a_{p;w}(pn+i))_{0 \leq i \leq p-1}$ are either constant (called type 1) or of the form
$$a_{p;w}(pn+i)=\begin{cases}
t^+ \;\; \text{if $i=w[|w|-1]$},\\
t\;\;\;\text{otherwise};
\end{cases}$$
for some integer $t \in \pset$ (called type 2).
Moreover, $(a_{p;w}(pn+i))_{0 \leq i \leq p-1}$ is of type 2 if and only if $w^\diamond$ is a suffix of $[n]_p$.\qed
\end{prop}

For the sequence $(a_{p;w}(n))_{n \in \mathbf{N}}$, let us define a $p$-block to be a sub-sequence of the form $(a_{p;w}(pn+i))_{0 \leq i \leq p-1}$ for some integer $n$. From the previous proposition, a $p$-block is either of type 1 or type 2. For a $p$-block $(a_{p;w}(pn+i))_{0 \leq i \leq p-1}$ of type 2, let us define its index to be an integer $i \in \pset$ such that $a_{p;w}(pn+i) \neq a_{p;w}(pn+j)$ for all $j \neq i$.

\begin{prop} 
\label{prop:multiple}
For any prime number $p$ and any $w \in \pset^*$, if there exists a word $v$ such that $v^{p+1}$ is a prefix of $(a_{p;w}(n))_{n \in \mathbf{N}}$ and that $|v| \geq 2p^{|w|}$, then $|v|$ is a multiple of $p^{|w|-1}$.
\end{prop}

\begin{proof}
If $|v| \geq 2p^{|w|}$, then, from Proposition \ref{prop:blockx} and \ref{prop:block0}, $v$ contains a $p$-block of the form 
$$a_{p;w}(pm+i)=\begin{cases}
1 \;\; \text{if $i=w[|w|-1]$},\\
0 \;\; \text{otherwise},
\end{cases}$$
for some $m$. Since $v^{p+1}$ is a prefix of $(a_{p;w}(n))_{n \in \mathbf{N}}$, $(a_{p;w}(pm+p|v|+i))_{0 \leq i \leq p-1}=(a_{p;w}(pm+i))_{0 \leq i \leq p-1}$, which is also a $p$-factor of type 2. From Proposition \ref{prop:block}, $w^\diamond$ is a suffix of both $[m]_p$ and $[m+|v|]_p$. Thus, $m+|v|-m$ is a multiple of $p^{|w|-1}$.
\qed
\end{proof}

\begin{prop} 
\label{prop:power}
For any prime integer $p$ and any $w \in \pset^*$, the sequence $(a_{p;w}(n))_{n \in \mathbf{N}}$ cannot have a prefix $v^{p+1}$ such that $|v|=ip^{|w|-1}$ for some positive integer $i \geq p+1$.
\end{prop}

This proposition will be proved with the help of the following lemmas.

\begin{lem} 
\label{lem:words-1}
Let $w\in \pset^*$, then for any words $a,b \in \pset^*$ and for any positive integer $\ell$, there exists a word $u$ such that $|u|=l$, $|au|_w=|a|_w$ and $|bu|_w=|b|_w$.
\end{lem}

\begin{proof}
Let $x\in\pset\backslash\{w[|w|-1]\}$ and $u=x^{\ell}$. It is clear that $|au|_w=|a|_w$ and $|bu|_w=|b|_w$ because none of the added factor of size $|w|$ ends with $x$.
\end{proof}

\begin{lem} 
\label{lem:words-2}
Let $w$ be a word in $\pset^*$ such that $|w| >1$. Let $a,b \in \pset^*$ such that $a_w\neq b_w$ where $a_w$ and $b_w$ are the longest suffixes of respectively $a$ and $b$ that are prefixes of $w$. Then there exists a word $u$ such that $|u|\leq |w|-1$ and that  $|au|_w \not \equiv |bu|_w \mod p$.
\end{lem}

\begin{proof}
If $|a|_w \not \equiv |b|_w \mod p$, then let $v= \epsilon$.\\
If $|a|_w \equiv |b|_w \mod p$, because have $a_w \neq b_w$, then $|a|_w\neq |b_w|$ because $w$ doesn't have multiple suffixes of the same length. Suppose that $a_w$ is the longest. It is clear that $|a_w|>0$. We define $v$ to be a word satisfying $a_wv=w$. In this case, $|v| \leq |w|-1$, $|av|_w=|a|_w+1$ and $|bv|_w=|b|_w$. 
\end{proof}

Now we are able to prove Proposition \ref{prop:power}.

\begin{proof} [of Proposition \ref{prop:power}]
We only need to prove that there exist $k, k'\in \pset$ such that $$(a_{p;w}(n))_{kip^{|w|-1}\leq n<(k+1)ip^{|w|-1}} \neq (a_{p;w}(n))_{k'ip^{|w|-1}\leq n<(k'+1)ip^{|w|-1}},$$
i.e. there exists some $j$ such that $0\leq j<|v|$ and $$a_{p;w}(kip^{|w|-1}+j) \neq a_{p;w}(k'ip^{|w|-1}+j).$$

For $1\leq k \leq p$, let $t_k=[kip^{|w|-1}]_p$. One has $t_k=u_kx_k0^j$ for some word $u_k$, some letter $x_k\in\pset\backslash\{0\}$ and some non-negative integer $j\geq |w|-1$.
Note that $u_1\neq 0$. Since $p$ is prime, one has $x_k\neq x_{k'}$ if $k\neq k'$. Thus, 
there exists $k\in \pset$ such that $x_k=w[0]$.

Now, let $k'\in \pset\backslash\{k\}$ and let $v_k$ and $v_k'$ be the longest suffixes of respectively $u_kx_k$ and $u_{k'}x_{k'}$ that are prefixes of $w$.
Because $x_k=w[0]$, $v_k\neq \epsilon$ and thus $v_k\neq v_{k'}$.
Therefore, by Lemma \ref{lem:words-2} and Lemma \ref{lem:words-1}, there exists $u$ such that $|v_ku|_w\not\equiv|v_{k'}u|_w \mod p$ and that $|u|=|w|-1$.
Let $j=[u]_p$, clearly $j<ip^{|w|-1}$ and one has $$a_{p;w}(kip^{|w|-1}+j) \neq a_{p;w}(k'ip^{|w|-1}+j),$$ which proves the result.
%
%
%
\end{proof}

Now we are able to prove the principle theorem in most cases:

\begin{thm}
\label{th1-weak}
For any prime number $p$ and any $w \in \pset^*$, the sequence $(a_{p;w}(n))_{n \in \mathbf{N}}$ is $p$-uniformly morphic for any $w$ and non-purely morphic when $|w|>1$ and $w\neq 10$.
\end{thm}

\begin{proof}
First, the fact that $(a_{p;w}(n))_{n \in \mathbf{N}}$ is $p$-automatic for any word $w$ follows from the Proposition 3.1 in \cite{cateland}, Page 7 and Theorem 16.1.5 in \cite{allouche_shallit_2003}.

Now, if $w\neq 10$ and $|w|>1$ and the sequence $(a_{p;w}(n))_{n \in \mathbf{N}}$ is purely morphic, then $0^{p+1}$ is a prefix of $(a_{p;w}(n))_{n \in \mathbf{N}}$.
Thus, $(a_{p;w}(n))_{n \in \mathbf{N}}$ will have infinitely many prefix of type $v^{p+1}$. However, from Proposition \ref{prop:multiple} and \ref{prop:power}, $(a_{p;w}(n))_{n \in \mathbf{N}}$ can only have finitely many prefix of the form
 $v^{p+1}$. We conclude.\qed
\end{proof}

Here we prove the $p$ particular cases.

\begin{prop}
\label{prop:k}
For any prime number $p$ and for any $w \in \pset \backslash\{0\}$, the sequence $(a_{p,w}(n))_{n \in \mathbf{N}}$ is purely morphic.
\end{prop}

\begin{proof}
It is easy to check that for any non-negative integer $m$, $(a_{p;w}(pm+i))_{0 \leq i \leq p-1}$ satisfies the following property:
$$a_{p;w}(pm+i)=\begin{cases}
a_{p;w}(m)^+ \;\; \text{if $i=w$},\\
a_{p;w}(m) \;\; \text{otherwise}.
\end{cases}$$
Thus, it is easy to check that $(a_{p;w}(pm+i))_{0 \leq i \leq p-1}$ is the fixed point of the morphism: $i \to v_i$ for all $i \in \pset$, where,
$$v_i[k]=\begin{cases}
i^+ \;\; \text{if $k=w$},\\
i \;\; \text{otherwise}.
\end{cases}$$
\qed
\end{proof}

\begin{prop}
\label{prop:0-2}
The sequence $(a_{2,0}(n))_{n \in \mathbf{N}}$ is non-purely morphic.
\end{prop}

\begin{proof}
The sequence $(a_{2,0}(n))_{n \in \mathbf{N}}$ begins with $1,0,1,0$. Thus, if this sequence is purely morphic, then this sequence has infinitely many prefixes of the form $v^2$. Here we prove that $(a_{2,0}(n))_{n \in \mathbf{N}}$ cannot have a prefix of the form $v^2$ with $|v| \geq 5$.

If $(a_{2,0}(n))_{n \in \mathbf{N}}$ has a prefix of the form $v^2$ with $|v| \geq 5$. Let us suppose that $|v|=4k+i$ for some non-negative integers $k,i$ such that $i=0,1,2,3$ and $k \geq 1$. First note that $a_{2,0}(r)=a_{2,0}(4k+i+r)$ for any $r$ satisfying $0 \leq r < 4k+i$.

Also note that, for $k\geq 0$, $a_{2,0}(k)=a_{2,0}(4k)=a_{2,0}(4k+3)=a_{2,0}(2k+1)=x$ and $a_{2,0}(4k+1)=a_{2,0}(4k+2)=a_{2,0}(2k)=x^+$ for some $x\in\{0,1\}$. This is because if $[k]_2=u$, then $[2k]_2=u0$, $[2k+1]_2=u1$,  $[4k]_2=u00$, $[4k+1]_2=u01$, $[4k+2]_2=u10$ and $[4k+3]_2=u11$.

Finally, note that $a_{2,0}(2^tk-1)=a_{2,0}(k-1)$, because, we know that $k\geq 1$ so if $[k-1]_2=u$ then $[2^tk-1]_2=u1^t$. Similarly, $a_{2,0}(2^tk-2^s-1)=a_{2,0}(2^tk-1)^+$ if $1<s<t$.

Now if $i=0$, then $|v|=4k$ and $a_{2,0}(1)=a_{2,0}(4k+1)=0$, $a_{2,0}(2)=a_{2,0}(4k+2)=1$ which contradicts to $a_{2,0}(4k+1) = a_{2,0}(4k+2)$.

If $i=1$, then $|v|=4k+1$ and $a_{2,0}(0)=a_{2,0}(4k+1)=1$, $a_{2,0}(1)=a_{2,0}(4k+2)=0$ which contradicts to $a_{2,0}(4k+2) = a_{2,0}(4k+2)$.

If $i=2$, then $|v|=4k+2$ and $a_{2,0}(k-1)=a_{2,0}(4k-1)$ but $a_{2,0}(k-1)=a_{2,0}(8k-1)=a_{2,0}(4k-3)=a_{2,0}(4k-1)^+$.


If $i=3$, then $|v|=4k+3$ and $a_{2,0}(2(2k+1))=a_{2,0}(4k+2)=a_{2,0}(4k+3)^+=a_{2,0}(0)^+=0$. Thus, we have $a_{2,0}(2k+1)=1$. But on the other hand, $a_{2,0}(2k+1)=a_{2,0}(4(2k+1))=a_{2,0}(8k+4)=a_{2,0}(4k+1)=a_{2,0}(4k+2)=0$, which is a contradiction.

%
 
In all cases, $(a_{2,0}(n))_{0 \leq n < |v|} \neq (a_{2,0}(n))_{|v| \leq n < 2|v|}$.
 \qed
\end{proof}

\begin{prop}
\label{prop:0-k}
For any prime number $p \geq 3$, the sequence $(a_{p,0}(n))_{n \in \mathbf{N}}$ is non-purely morphic.
\end{prop}

\begin{proof}
The sequence $(a_{p,0}(n))_{n \in \mathbf{N}}$ begins with $(10^{p-1})^p$. Thus, if this sequence is purely morphic, then this sequence has infinitely many prefixes of the form $v^2$. Here we prove that $(a_{p,0}(n))_{n \in \mathbf{N}}$ cannot have a prefix of the form $v^2$ with $|v| \geq p^2$.

First, let us prove that if $v^2$
is a prefix of $(a_{p,0}(n))_{n \in \mathbf{N}}$, then $|v|$ is
a multiple of $p$. It is easy to check that $v^2$ is not a prefix of $(a_{p,0}(n))_{n \in \mathbf{N}}$ when $|v|=1,2$. Let us suppose that $|v| \geq 3$. In this case, $v$ begins with $1, 0, 0$. Thus, $a_{p,0}(|v|) = 1,
a_{p,0}(|v| + 1) = a_{p,0}(|v| + 2) = 0$.

Let us suppose that $|v|=kp+t$ for some nonnegative integers $k,t$ such that $0 \leq t \leq k-1$. We first prove that $t\neq k-1$. If it is the case, then $|v|+1$ is a multiple of $p$ and $a_{p,0}(|v| + 1) \neq a_{p,0}(|v| + 2) = 0$ since $|v|+2$ has one $0$ less than $|v|+1$ in their p-expansions. this contradicts the fact that $a_{p,0}(|v| + 1) = a_{p,0}(|v| + 2) = 0$.

Now let us suppose that $t \neq k-1$. In this case, $(a_{p,0}(n))_{kp\leq n \leq (k+1)p-1}$ contains the factor $a_{p,0}(|v|)a_{p,0}(|v|+1)=1,0$. Thus, $(a_{p,0}(n))_{kp\leq n \leq (k+1)p-1}$ is a factor of type 2 announced in proposition 8. But the word $(a_{p,0}(n))_{0\leq n \leq p-1}=10^{p-1}$ is also a word of type 2 and the word $(a_{p,0}(n))_{n \in \mathbf{N}}$ cannot have two different factors of type 2 such that the special letters are at different positions. Thus, $ a_{p,0}(|v|)a_{p,0}(|v|+1)$ should be a prefix of $(a_{p,0}(n))_{kp\leq n \leq (k+1)p-1}$ and consequently $|v|$ is a multiple of $p$. 


Second, $|v|$ is not a multiple of $p^2$. Because, if it is in this case, $a_{p,0}(|v|)=a_{p,0}(0)=1$ and $a_{p,0}(|v|+p)=a_{p,0}(|v|)-1=0$. But $a_{p,0}(p)=1$, thus, $a_{p,0}(|v|+p) \neq a_{p,0}(p)$. Consequently, $(a_{p,0}(n))_{0 \leq n <|v|-1} \neq (a_{p,0}(n))_{|v| \leq n <2|v|-1}$.

Third, if $|v|$ is not a multiple of $p^2$ but larger than $p^2+1$, let us suppose that $|v|=kp^2+tp$ for some positive integers $k,t$ such that $1 \leq t \leq p-1$. Let  $x=(p-t)p$, we then have $a_{p,0}(|v|+x)=a_{p,0}(x)=1$. But in this case, $a_{p,0}(x+p)=1$ or $2$, but $a_{p,0}(|v|+x+p)=0$. Thus, $(a_{p,0}(n))_{0 \leq n <|v|-1} \neq (a_{p,0}(n))_{|v| \leq n <2|v|-1}$.
 \qed
\end{proof}

\begin{prop}
\label{prop:10-2}
The sequence $(a_{2;10}(n))_{n \in \mathbf{N}}$ is non-purely morphic.
\end{prop}

\begin{proof}
The sequence $(a_{2;10}(n))_{n \in \mathbf{N}}$ begins with $0010$. Thus, if this sequence is purely morphic, then this sequence has infinitely many prefixes of the form $v^2$. We will prove that its only prefix of square shape is $00$.

Let $v^2$ be a prefix of $(a_{p;10}(n))_{n \in \mathbf{N}}$. Because of that, one can note that $(a_{p;10}(n))_{0\leq n <|v|} = a_{p;10}(|v|+n))_{0\leq n<|v| }$, in particular, $(a_{p;10}(|v|+n))_{0\leq n \leq 4}=0010$. Using this, we will prove this proposition by proving all the different possibility for the word $[|v|]_2$.

For now on, $u$ can be any word in $\pset^*$ and $s$ and $t$ positive integer. Note that the computation is made in binary basis.

i) If $[|v|]_2=1^t$ with $t>1$, we have $a_{p;10}(|v|+1)=1\neq 0$ because $1^t + 1 = 10^t$.

ii) If $[|v|]_2=1^t01$, one can simply note that $a_{p;10}(|v|)=1\neq0$.

iii) If $[|v|]_2=u101^t01$ then $a_{p;10}(|v|+3)=a_{p;10}(|v|)^+$ because $u101^t01+11 = u110^{s+2}$.

iv) If $[|v|]_2=u101^t$ with $t>1$ then $a_{p;10}(|v|+2)=a_{p;10}(|v|)$, because $u101^t + 11 = u110^{t-1}1$.

v) If $[|v|]_2=u10^s1^t$ with $s>1$ we have $a_{p;10}(|v|+1)=a_{p;10}(|v|)^+$, because $u10^s1^t + 1 = u10^{s-1}10^t$.

vi) Finally, if $[|v|]_2=u10^t$ with $t>1$, we have on one hand $a_{p;10}(|v|+(1^{t-1}0)_2)=a_{p;10}(|v|)=0$ because $u10^t+1^{t-1}0=u1^t0$. We also have $a_{p;10}(|v|+(1^{t-1}0)_2)=a_{p;10}((1^{t-1}0)_2)=1$, which is a contradiction.

An attentive reader will remark that this cover all the number strictly bigger than 1.
\qed
\end{proof}

\begin{prop}
\label{prop:10-p}
For any prime number $p\geq 3$  the sequence $(a_{p;10}(n))_{n \in \mathbf{N}}$ is non-purely morphic.
\end{prop}

\begin{proof}
The sequence $(a_{p;10}(n))_{n \in \mathbf{N}}$ begins with $0^p1$. Thus, if this sequence is purely morphic, then this sequence has infinitely many prefixes of the form $v^p$. It suffices to prove that if $|v|>p^2$ then $v^p$ is not a prefix of $(a_{p;10}(n))_{n \in \mathbf{N}}$.

Let $v^p$ be a prefix of $(a_{p;10}(n))_{n \in \mathbf{N}}$.

Suppose that $p\nmid |v|>p^2$. This means that $v=uyx$ for a word $u$ and some letters $y,x$ with $x\neq 0$. Because $v^2$ is a prefix of $(a_{p;10}(n))_{n \in \mathbf{N}}$, $v$ begins with the letters $0^p1$ and $(a_{p;10}(n))_{0\leq n \leq p} = a_{p;10}((v)_p+n))_{0\leq n \leq p}$. Thus $a_{p;10}((v)_p)=0$.

Let $c\in\pset$ such that $x+c=p$; it exists because $x\neq 0$ and $p>2$. Thus, $[(v)_p +c]_p=u'y'0$. Because $a_{p;w}(c)=0$, $a_{p;w}((v)_p +c)=0$ also and $a_{p;w}((v)_p +p)=0$ or $p-1$ which is not equal to $a_{p;w}(p)=1$. Therefore, $v^2$ is not a prefix of $(a_{p;10}(n))_{n \in \mathbf{N}}$.

Suppose now that $p\mid |v|\geq p$. Let $|v|=sp^t$ for some positive integer $s,t$ such that $t \geq 1$ and $p \nmid s$ and let $[v]_p=ux0^{t}$ for some word $u$ and some letter $x\in\pset\backslash\{0\}$.

Since $p$ is prime, there exists $k\in\pset$ such that $[kv]_p=u'10^{t}$ for some word $u'$. Let $m=p^{t+1}-1$, thus $[m]_p=(p-1)^t$ and $[kv+m]_p=u'1(p-1)^t$.

Since $(a_{p;10}(n))_{k_1|v| \leq n \leq (k_1+1)|v|-1}= (a_{p;10}(n))_{k_2|v| \leq n \leq (k_2+1)|v|-1}$, for any $k_1$, $k_2\in\pset$ we have $a_{p;10}(0)=0=a_{p;10}(kv)$ thus $a_{p;10}(u')=p-1$ which means that $a_{p;10}(m)=0\neq a_{p;10}(kv+m)=p-1$.

Hence, $v^p$ cannot be a prefix of $(a_{p;10}(n))_{n \in \mathbf{N}}$ if $v>p^2$ which concludes the proof.
\qed
\end{proof}

\begin{proof}[of Theorem~\ref{th1}]
It is a direct result of Theorem~\ref{th1-weak}, Proposition~\ref{prop:k}, Proposition~\ref{prop:0-2}, Proposition~\ref{prop:0-k}, Proposition~\ref{prop:10-2} and Proposition~\ref{prop:10-p}.\qed

\end{proof}

\section{Algebraicity}
\label{sec:alg}

By Christol's theorem \cite{Christol1980KMFR}, we know that the power series
$f=\sum_{i=0}^{\infty} a_{p;w}(n)t^n$ is algebraic over $\mathbb{F}_p(t)$. 
Now we prove that $f$ is algebraic of degree $p$. Indeed,
if we let $[w]_p$ denote $w_1p^{k-1}+\cdots +w_k$,
and write $a_n=a_{p;w}(n)$ for short, then
\begin{align*}
	&\quad (1+t+\cdots +t^{p-1}) f^p -f\\
	&=\sum_{n\geq 0} \sum_{j=0}^{p-1} (a_n - a_{pn+j})t^{pn+j}\\
	&=\sum_{n\geq 0} (a_n - a_{pn+w_k})t^{pn+w_k}\\
	&=\sum_{n\geq 0}\sum_{j=0}^{p-1} (a_{np+j} - a_{np^2+jp+w_k})t^{np^2+jp+w_k}\\
	&=\sum_{n\geq 0}(a_{np+w_{k-1}} - a_{np^2+w_{k-1}p+w_k})t^{np^2+w_{k-1}p+w_k}\\
	&\ldots\\
	&=\sum_{n\geq 0} (a_{np^{k-1}+ w_1 p^{k-2}  \cdots +w_{k-1}}-a_{np^{k}+w_1p^{k-1}+\cdots + w_k}) t^{np^{k}+[w]_p}\\
 &=\begin{cases}  
	 \sum_{n\geq 0} - t^{np^{k}+[w]_p} =t^{[w]_p} /(t^{p^{k}}-1), \; &\text{if } w_1\neq 0\\
	 \sum_{n\geq 1} - t^{np^{k}+[w]_p}=t^{p^{k}+[w]_p} /(t^{p^{k}}-1),\; &\text{if } w_1= 0.
  \end{cases}
\end{align*}

The irreduciblity of the the above functional equations is straightforward from the Eisenstein's criterion. We thus have the following propriety:

\begin{prop}
For any prime number $p$ and any finite word $w$ in $[\![p]\!]^*$, the power series $\sum_{i=0}^{\infty} a_{p;w}(n)t^n$ is algebraic of degree $p$ over $\mathbb{F}_p(t)$.
\end{prop}

\section{Final remarks}
\label{sec:rem}

The authors remark that the fast algorithms introduced in Section \ref{sec:windows} for $0$-words and non-$0$-words are much different. However, the generating functions given in Section \ref{sec:alg} for $0$-words and non-$0$-words are quite similar. Thus, we believe that the algorithms in Section \ref{sec:windows} can be unified for both $0$-words and non-$0$-words.


\bibliographystyle{splncs03}
\bibliography{biblio}

\end{document}